\documentclass[12pt]{amsart}
\usepackage{amssymb}
\usepackage{verbatim}
\usepackage{hyperref}
\usepackage{color}

\begin{document}

\newtheorem{theorem}{Theorem}    
\newtheorem{proposition}[theorem]{Proposition}
\newtheorem{conjecture}[theorem]{Conjecture}
\def\theconjecture{\unskip}
\newtheorem{corollary}[theorem]{Corollary}
\newtheorem{lemma}[theorem]{Lemma}
\newtheorem{sublemma}[theorem]{Sublemma}
\newtheorem{observation}[theorem]{Observation}
\theoremstyle{definition}
\newtheorem{definition}{Definition}
\newtheorem{notation}[definition]{Notation}
\newtheorem{remark}[definition]{Remark}
\newtheorem{question}[definition]{Question}
\newtheorem{questions}[definition]{Questions}
\newtheorem{example}[definition]{Example}
\newtheorem{problem}[definition]{Problem}
\newtheorem{exercise}[definition]{Exercise}

\numberwithin{theorem}{section} \numberwithin{definition}{section}
\numberwithin{equation}{section}

\def\earrow{{\mathbf e}}
\def\rarrow{{\mathbf r}}
\def\uarrow{{\mathbf u}}
\def\varrow{{\mathbf V}}
\def\tpar{T_{\rm par}}
\def\apar{A_{\rm par}}

\def\reals{{\mathbb R}}
\def\torus{{\mathbb T}}
\def\heis{{\mathbb H}}
\def\integers{{\mathbb Z}}
\def\naturals{{\mathbb N}}
\def\complex{{\mathbb C}\/}
\def\distance{\operatorname{distance}\,}
\def\support{\operatorname{support}\,}
\def\dist{\operatorname{dist}\,}
\def\Span{\operatorname{span}\,}
\def\degree{\operatorname{degree}\,}
\def\kernel{\operatorname{kernel}\,}
\def\dim{\operatorname{dim}\,}
\def\codim{\operatorname{codim}}
\def\trace{\operatorname{trace\,}}
\def\Span{\operatorname{span}\,}
\def\dimension{\operatorname{dimension}\,}
\def\codimension{\operatorname{codimension}\,}
\def\nullspace{\scriptk}
\def\kernel{\operatorname{Ker}}
\def\ZZ{ {\mathbb Z} }
\def\p{\partial}
\def\rp{{ ^{-1} }}
\def\Re{\operatorname{Re\,} }
\def\Im{\operatorname{Im\,} }
\def\ov{\overline}
\def\eps{\varepsilon}
\def\lt{L^2}
\def\diver{\operatorname{div}}
\def\curl{\operatorname{curl}}
\def\etta{\eta}
\newcommand{\norm}[1]{ \|  #1 \|}
\def\expect{\mathbb E}
\def\bull{$\bullet$\ }
\def\C{\mathbb{C}}
\def\R{\mathbb{R}}
\def\Rn{{\mathbb{R}^n}}
\def\Sn{{{S}^{n-1}}}
\def\M{\mathbb{M}}
\def\N{\mathbb{N}}
\def\Q{{\mathbb{Q}}}
\def\Z{\mathbb{Z}}
\def\F{\mathcal{F}}
\def\L{\mathcal{L}}
\def\S{\mathcal{S}}
\def\supp{\operatorname{supp}}
\def\dist{\operatorname{dist}}
\def\essi{\operatornamewithlimits{ess\,inf}}
\def\esss{\operatornamewithlimits{ess\,sup}}
\def\xone{x_1}
\def\xtwo{x_2}
\def\xq{x_2+x_1^2}
\newcommand{\abr}[1]{ \langle  #1 \rangle}

\newcommand{\Norm}[1]{ \left\|  #1 \right\| }
\newcommand{\set}[1]{ \left\{ #1 \right\} }
\def\one{\mathbf 1}
\def\whole{\mathbf V}
\newcommand{\modulo}[2]{[#1]_{#2}}

\def\scriptf{{\mathcal F}}
\def\scriptg{{\mathcal G}}
\def\scriptm{{\mathcal M}}
\def\scriptb{{\mathcal B}}
\def\scriptc{{\mathcal C}}
\def\scriptt{{\mathcal T}}
\def\scripti{{\mathcal I}}
\def\scripte{{\mathcal E}}
\def\scriptv{{\mathcal V}}
\def\scriptw{{\mathcal W}}
\def\scriptu{{\mathcal U}}
\def\scriptS{{\mathcal S}}
\def\scripta{{\mathcal A}}
\def\scriptr{{\mathcal R}}
\def\scripto{{\mathcal O}}
\def\scripth{{\mathcal H}}
\def\scriptd{{\mathcal D}}
\def\scriptl{{\mathcal L}}
\def\scriptn{{\mathcal N}}
\def\scriptp{{\mathcal P}}
\def\scriptk{{\mathcal K}}
\def\frakv{{\mathfrak V}}

\title[Some one-sided estimates for oscillatory singular integrals]
{Some one-sided estimates for oscillatory singular integrals}
\author[Zunwei Fu]{Zunwei Fu} 

\author[Shanzhen Lu]{Shanzhen Lu}

\author[Yibiao Pan]{Yibiao Pan}

\author[Shaoguang Shi]{Shaoguang Shi}

%
\thanks{This work was partially supported by
 NSF of China (Grant Nos. 11271175, 10931001 and 11301249), NSF of Shandong Province (Grant No. ZR2012AQ026), the Key Laboratory of Mathematics and Complex System (Beijing Normal
University), the Fundamental Research Funds for the Central Universities (Grant No. 2012CXQT09) and the AMEP(DYSP) of Linyi University. \\ \indent }
\address{Zunwei Fu\\Department of Mathematics\\Linyi
University \\
Linyi 276005\\
P. R. China } \email{zwfu@mail.bnu.edu.cn}

\address{Shanzhen Lu\\School of Mathematical
Sciences\\Beijing Normal University\\Beijing, 100875\\
P. R. China} \email{lusz@bnu.edu.cn}

\address{Yibiao Pan\\Department of Mathematics\\
University of Pittsburgh \\
Pittsburgh, PA 15260\\
 U.S.A.}
\email{yibiao@pitt.edu}

\address{Shaoguang Shi\\Department of Mathematics\\
Linyi
University \\
Linyi 276005\\
P. R. China}
\email{shishaoguang@mail.bnu.edu.cn}


\maketitle

\begin{abstract}
The purpose of this paper is to establish some one-sided estimates for oscillatory singular integrals. The boundedness of certain oscillatory singular integral on weighted Hardy spaces $H^{1}_{+}(w)$ is proved. It is here also show that the $H^{1}_{+}(w)$ theory of oscillatory singular integrals above cannot be extended to the case of $H^{q}_{+}(w)$ when $0<q<1$ and $w\in A_{p}^{+}$, a wider weight class than the classical Muckenhoupt class. Furthermore, a criterion on the weighted $L^{p}$-boundednesss of the oscillatory singular integral is given.
\end{abstract}

\section{Introduction and main results}\label{section1} 

The study of one-sided operators was motivated not only as the generalization
of the theory of two-sided ones but also by the demand in ergodic theory \cite{C}, \cite{DS}. The well-known
Riemann-Liouville fractional integral can be viewed as the
one-sided version of Riesz potential \cite{MR}.
In
\cite{Saw}, Sawyer studied the weighted theory of one-sided maximal Hardy-Littlewood
operators in depth for the first time. Since then, numerous papers have appeared,
among which we choose to refer to \cite{AFM}, \cite{AS}, \cite{FLSS}, \cite{Ma}, \cite{MOT} and \cite{RS2.5} about one-sided
operators, \cite{AC}, \cite{FL}, \cite{OST}, \cite{RiT}, \cite{RS1}, and \cite{RS3} about one-sided spaces, respectively. Interestingly, lots of
results show that for a class of smaller operators (one-sided operators) and a class
of wider weights (one-sided weights), many famous results in harmonic analysis still
hold.

Besides the Hardy-Littlewood maximal operators and Calder\'{o}n-Zygmund singular integral operators, oscillatory integral operators have
played an important role in harmonic analysis from its outset; three chapters are devoted to them in the celebrated Stein's book \cite{St}. Many important operators in harmonic analysis are some versions of oscillatory integrals, such as Fourier transform, Bochner-Riesz means, Radon
transform in CT technology and so on. For a more complete account on
oscillatory integrals in classical harmonic analysis, we would like to refer the interested reader to \cite{G}, \cite{L2}, \cite{LDY}, \cite{LZ1}, \cite{LZ2},  \cite{PS}, \cite{RS}, \cite{Sa} and references therein. In more recent times, the operators fashioned from oscillatory integrals, such
as pseudo-differential operator in PDE become another motivation to study
them. Based on the estimates of some kinds of oscillatory integrals, one can
establish the well-posedness theory of a class of dispersive equations,
for some of these works, we refer to \cite{CM}, \cite{KPV1} and \cite{KPV2}.

Inspired by theory of oscillatory singular operators and one-sided operators, the authors of this paper defined the one-sided oscillatory integral operator in \cite{FLSS} (see also \cite{FSL}), which is just the object of this paper. We first recall its definitions as
$$
T^{+}f(x)=\lim_{\varepsilon\rightarrow0^{+}}\int_{x+\varepsilon}^{\infty}e^{iP(x,y)}K(x-y)f(y)dy=\mathrm{p.v.}\int_{x}^{\infty}e^{iP(x,y)}K(x-y)f(y)dy
$$
and
$$
T^{-}f(x)=\lim_{\varepsilon\rightarrow0^{+}}\int_{-\infty}^{x-\varepsilon}e^{iP(x,y)}K(x-y)f(y)dy=
\mathrm{p.v.}\int_{-\infty}^{x}e^{iP(x,y)}K(x-y)f(y)dy,
$$
where $P(x,y)$ is a real polynomial defined on
$\mathbb{R}\times\mathbb{R}$, and $K$ is a one-sided
Calder\'{o}n-Zygmund kernel with support in
$\mathbb{R}^{-}=(-\infty,0)$ and $\mathbb{R}^{+}=(0,+\infty)$,
respectively. We recall that a function $K\in L_{loc}^{1}\{\mathbb{R}/{\{0\}}\}$ is a Calder\'{o}n-Zygmund kernel, if
there exists a finite constant C such that the following properties are satisfied:

(a)
$$
\left|\int_{a<|x|<b}K(x)\, dx\right|\leq C, \quad 0<a<b
$$
holds for all $a$ and $b$, and there exists $\lim_{\varepsilon\rightarrow 0^{+}}\int_{\varepsilon<|x|<1}K(x)dx,$

(b) $$
|K(x)|\leq{C}/{|x|},  \,\,x\neq 0,
$$

(c)
$$|K(x - y)-K(x)| \leq {C|y|}/{|x|^{2}}$$
for all $x$ and $y$ with $|x|>2|y|>0$.

A Calder\'{o}n-Zygmund kernel with support in $(-\infty, 0)$ (or in $(0, \infty)$) will be called the one-sided
Calder\'{o}n-Zygmund kernel. In \cite{AFM}, Aimar, Forzani and Mart\'{\i}n-Reyes give an example of
such kernel
$$K(x)=\frac{\sin(\log|x|)}{(x\log|x|)}\chi_{(-\infty,0)}(x),$$
where $\chi_{E}$ denotes the characteristic function on a set $E$.

In order to give the main results of our paper, some definitions and propositions for one-sided weights are needed.
Let $f(x)$ be a measurable function defined on $\mathbb{R}$. The one-sided Hardy-Littlewood maximal functions $M^{+}f(x)$ and $M^{-}f(x)$ are defined by
$$M^{+}f(x)=\sup_{h>0}\frac{1}{h}\int_{x}^{x+h}|f(y)|\,dy$$
and
$$M^{-}f(x)=\sup_{h>0}\frac{1}{h}\int_{x-h}^{x}|f(y)|\,dy,$$
which arose in the ergodic  maximal function, see \cite{Saw}.

As usual, a weight $w(x)$ is a measurable and non-negative function. If $O\subset \mathbb{R}$ is a Lebesgue measurable set, we denote its $w$-measure by $w(O)=\int_{O}w(t)dt$. A function $f(x)$ belongs to $L^{p}(w)(0<p\leq\infty)$, if $\|f\|_{L^{p}(w)}=(\int_{\mathbb{R}}|f(x)|^{p}w(x)dx)^{1/p}<\infty$. If $w(x)=1$, we denote $\|f\|_{L^{p}(w)}$ simply by $\|f\|_{L^{p}(\mathbb{R})}$.
The classical Dunford-Schwartz ergodic theorem (see \cite{DS}) can be considered as the first result about weights (one-sided weights) for $M^{+}$ and $M^{-}$. A weight $w(x)$ belongs to the class $A_{p}^{+}$, $A_{p}^{-}$ (one-sided $A_{p}$ weights) defined by Sawyer \cite{Saw}, if they satisfy the following conditions:
$$
 A_{p}^{+}(w)=:\sup_{a<b<c}\frac{1}{(c-a)^{p}}\int_{a}^{b}w(x)\,dx\left(\int_{b}^{c}w(x)^{1-p'}\,dx\right)^{p-1}<\infty,
$$
$$
 A_{p}^{-}(w)=:\sup_{a<b<c}\frac{1}{(c-a)^{p}}\int_{b}^{c}w(x)\,dx\left(\int_{a}^{b}w(x)^{1-p'}\,dx\right)^{p-1}<\infty,
$$
 where $1<p<\infty$. When $p=1$ and $p=\infty$,
\begin{equation}\label{1.1}
A_{1}^{+}:\quad M^{-}w\leq Cw,\quad A_{1}^{-}:\quad M^{+}w\leq Cw,
\end{equation}
for some constant $C$ and
$$
A_{\infty}^{+}=\cup_{p>1}A_{p}^{+},\quad A_{\infty}^{-}=\cup_{p>1}A_{p}^{-}.
$$
The smallest constant $C$ for which (\ref{1.1}) is satisfied will be denoted by $A_{1}^{+}(w)$ and $A_{1}^{-}(w)$.
$A_{p}^{+}(w)$ (or $A_{p}^{-}(w)$), $p\geq 1$, will be called the $A_{p}^{+}$ (or $A_p^-$) constant of $w$.
Here
$A_{p}(1\leq p\leq\infty)$ denotes the Muckenhoupt class \cite{Mu}.
This class consists of weight functions $w$ for which
$$A_{p}:\sup_I\left(\frac{1}{|I|}\int_{I}w(x)dx\right)\left(\frac{1}{|I|}\int_{I}w(x)^{1-p'}dx\right)^{p-1}<\infty,\quad
1< p<\infty,$$
$$A_{1}:Mw\leq Cw,$$ and $$\quad A_{\infty}=\cup_{p>1}A_{p},
$$
where the supremum is taken over all intervals $I\subset\mathbb{R}$, $1/p+1/p'=1$ and $M$ is the classical Hardy-Littlewood maximal function $Mf(x)=\sup_{h>0}\frac{1}{h}\int_{x-h}^{x+h}|f(y)|\,dy$. Throughout this paper, the letter $C$ is used for various constants, and may change from one occurrence to another.

Given $w(x)\in A_{p}^{+}$, $1\leq p<\infty$, one can define $x_{-\infty}$ and $x_{+\infty}$  with $-\infty\leq x_{-\infty}\leq x_{+\infty}\leq \infty$ (see \cite{RS2}), such that

(a)\,\,$w(x)=0,$ \quad $x\in (-\infty, x_{-\infty});$

(b)\,\,$w(x)>0,$ \quad $x\in (x_{-\infty}, +\infty).$

In \cite{Saw}, Sawyer obtained the characterizations of the weighted inequalities for $M^{+}$ and $M^{-}$, respectively. He proved that $M^{+}$ was bounded on $L^{p}(w)$ if and only if $w\in A_{p}^{+}$ with $1<p<\infty$. For $p=1$, Sawyer also showed the weak weighted estimate for $M^{+}$ holds if and only if $w\in A_{1}^{+}$. Let us remark here that similar results can be obtained
for the left-hand-side operator (for example, $M^{-}$) by changing the condition
$A_{p}^{+}$ by $A_{p}^{-}$. Together with the characterizations of the weighted inequalities
for $M^{+}$ and $M^{-}$, Sawyer obtained some properties of the
classes $A_{p}^{+}$ and $A_{p}^{-}$:

(a)\,\, If $w\in A_{1}^{+}$, then $w^{1+\varepsilon}\in A_{1}^{+}$
for some $\varepsilon>0$.

(b)\,\, If $1\leq
p<\infty$, then $A_{p}=A_{p}^{+}\bigcap A_{p}^{-}$, $A_{p}\subset
A_{p}^{+}$, $A_{p}\subset A_{p}^{-}$.

(c)\,\, $A_p^+\subset
A_r^+$, $A_p^-\subset A_r^-$ if $1\leq p\leq r$.

We conclude from (b) that $A_{p}^{+}$ class is a wider class than $A_{p}$ class. Take $e^{x}$ for example, $e^{x}\notin A_{1}$, but $e^{x}\in A_{1}^{+}$. For more information about weights for one-sided operators, we refer the reader to a survey article of Mart\'{i}n-Reyes, Ortega and Torre \cite{MOT}.

Since $T^{+}$ is not a Calder\'{o}n-Zygumund operator, we cannot directly use the Calder\'{o}n-Zygumund theory to study this boundedness. Also, since $T^{+}$ is not a convolution operator, we cannot use Fourier transform method either. The study of the boundedness of $T^{+}$ centers on three main
questions which we shall describe below:

When $p=1$, both $M^{+}$ and the one-sided Calder\'{o}n-Zygumund singular integral operator
$\widetilde{T}^{+}$ which defined by
$$
\widetilde{T}^{+}f(x)=\mathrm{p.v.}\int_{x}^{\infty}K(x-y)f(y)dy=\lim_{\varepsilon\rightarrow0^{+}}\int_{x+\varepsilon}^{\infty}K(x-y)f(y)\,dy
$$
 are bounded from $L^{1}(w)$ to $L^{1,\infty}(w)$ with $w\in A_{1}^{+}$, where $L^{1,\infty}(w)$ denotes the weak $L^{1}(w)$ space with the seminorm $\|f\|_{L^{1,\infty}(w)}:= \sup_{\lambda>0}\lambda w(\{x\in \mathbb{R}: |f(x)|>\lambda\})$, see \cite{MOT}. Therefore, one naturally wants to know

\medskip
\noindent\textbf{Question 1}\quad Is $T^{+}$ bounded from $L^{1}(w)$ to $L^{1,\infty}(w)$ with $w\in A_{1}^{+}$?
\medskip

We answer Question 1 in \cite{FLSS} and showed $T^{+}$ maps $L^{1}(w)$ into
$ L^{1,\infty}(w)$ if $w\in A_{1}^{+}$.

The study of one-sided spaces emerged naturally alongside the study of one-sided operators. In one previous study, the authors studied one-sided BMO spaces associated with one-sided sharp functions and their relationship to good weights for the one-sided Hardy-Littlewood maximal functions \cite{MT2}.
It is well known that the classical Hardy spaces were the dual spaces of BMO spaces and were the natural alternative for Lebesgue spaces when $p<1$. For some classical work on classical Hardy spaces, we refer to \cite{DH}, \cite{FS1}, \cite{FS2}, \cite{HLLW} and \cite{L1}. In \cite{RS2}, Rosa and Segovia introduced the one-sided Hardy space $H^{q}_{+}(0<q\leq 1)$. We first recall its definition. As usual, $C_{0}^{\infty}(\mathbb{R})$ is the set of all functions with compact support having derivatives of all orders. For $-\infty\leq r<\infty$, we shall denote by $\mathcal{S}(r,\infty)$ the space of all $C_{0}^{\infty}(\mathbb{R})$ functions with support contained in $(r,\infty)$ equipped with the usual topology and by $\mathcal{S}'(r,\infty)$ the space of distributions on $(r,\infty)$. Given an integer $\gamma\geq 1$ and $x\in \mathbb{R}$, we shall say that a $C_{0}^{\infty}(\mathbb{R})$ function $\psi(t)$ belongs to the class $\Phi_{\gamma}(x)$ if there exists a bounded interval $I_{\psi}=[x,\beta]$ containing the support of $\psi(t)$ such that $D^{\gamma}\psi(t)$ satisfies $$
|I_{\psi}|^{\gamma+1}\|D^{\gamma}\psi\|_{L^{\infty}}\leq 1.
$$
Let $f$ be a distribution in $\mathcal{S}'(r,\infty)$. One defines the one-sided maximal function $M^{1}_{+,\gamma}(x)$ as
$$
M^{1}_{+,\gamma}(x)=\sup\{|\langle f,\psi\rangle|:\psi\in \Phi_{\gamma}(x)\}
$$
for every $x>\gamma$. We observe that $M^{1}_{+,\gamma}(x)$ is a lower semicontinuous function.

Let $w\in A_{p}^{+}$ and $0<q\leq 1$. For every integral $\gamma\geq 1$ satisfying $(\gamma+1)q\geq p>1$ or $(\gamma+1)q>1$ if $p=1$, we shall say that the distribution $f$ in $\mathcal{S}'(x_{-\infty},\infty)$ belongs to $H^{q}_{+,\gamma}(w)$ if the ``$p$-norm"
$$
\|f\|_{H^{q}_{+,\gamma}(w)}=\left(\int_{x_{-\infty}}^{\infty} (M^{1}_{+,\gamma}(x))^{q}w(x)dx\right)^{1/q}<\infty.
$$
It is easy to say that $H^{q}_{+,\gamma}(w)$ is a Banach space \cite{RS2}.

If $\psi\in \mathcal{S}$, $\supp(\psi)\subset (-\infty, 0]$, $\int_{-\infty}^{0}\psi(x)dx=0$, $f\in \mathcal{S}(x_{-\infty},\infty)$ and $x>x_{-\infty}$, then one can define another one-sided maximal functions
$$
M^{2}_{+,\psi}(x)=\sup_{t>0}|f\ast\psi_{t}(y)|,
$$
where $\psi_{t}(x)=t^{-1}\psi(x/t)$. We shall say that $f$ in $\mathcal{S}'(x_{-\infty},\infty)$ belongs to $H^{q}_{+}(w)$ if
$$
\|f\|_{H^{q}_{+}(w)}=\left(\int_{x_{-\infty}}^{\infty} (M^{2}_{+,\psi}(x))^{q}w(x)dx\right)^{1/q}<\infty,
$$
again, $0<q\leq 1$ and $w\in A_{p}^{+}(p\geq 1)$. In \cite[Theorem C]{RS3}, the authors proved that
$$H^{q}_{+}(w)=H^{q}_{+,\gamma}(w),$$
where $w\in A_{p}^{+}(p\geq 1)$, $0<q\leq 1$ and $q(\gamma+1)>p$.

There are still atomic decomposition for functions in $H^{q}_{+}(w)(0<q\leq 1)$. We first recall the definition of $H^{q}_{+}(w)$ atom \cite{RS2}. A function $a(x)$ defined on $\mathbb{R}$ is called a $q$-atom with respect to $w(x)$ if there exists an interval $I$ (not necessary bounded) containing the support of $a(x)$ such that

(a)\,\,$I\subset (x_{-\infty},\infty)$ and $w(I)<\infty,$

(b)\,\,$\|a\|_{L^{\infty}}\leq w(I)^{-1/q},$

(c)\,\,$|I|<dist(x_{-\infty}, I)$, and $\int_{I}a(x)dx=0$.

We shall say that $I$ is the interval associated to the atom $a(x).$

Let $w\in A_{p}^{+}(p\geq 1)$, $\gamma\geq 1$ be an integer and $0<q\leq 1$ such that $(\gamma+1)q\geq p>1$ or $(\gamma+1)q>1$ if $p=1$. Then, if
$f\in H^{q}_{+}(w)$, there exists a sequence $\{a_{k}(x)\}$ of $q-$atom with respect to $w(x)$ and a sequence $\{\lambda_{k}\}$ of real numbers such that
\begin{equation}\label{2.1}
f=\sum \lambda_{k}a_{k}(x) \in \mathcal{S}'(x_{-\infty},\infty)
\end{equation}
and
$$
\|f\|_{H^{q}_{+}(w)}\approx \sum |\lambda_{k}|^{q}.
$$
The sum in $(\ref{2.1})$ is both in the sense of distributions and in the $H^{q}_{+}(w)$ norm \cite[Theorem 2.2]{RS2}.

Besides one-sided maximal functions, Ombrosi and Segovia \cite{OS} studied the boundedness of the one-sided Calderon-Zygmund operator $\widetilde{T}^{+}$ on $H^{q}_{+}(w)(0<q\leq 1)$ with $w\in A_{p}^{+}(p\geq 1)$ under a generic condition and proved that $\widetilde{T}^{+}$ can be extended to bounded operators from $H^{q}_{+}(w)$ to $H^{q}_{+}(w)$. In fact, they proved the boundedness in a more general case, see \cite{OS} for more details.

It is easy to check that $H_{+}^{1}(w) \subset L^{1}(w)$. Therefore, if $T^{+}$ maps $H_{+}^{1}(w)$ into $L^{1}(w)$, then we can prove the weighted $L^{p}$ boundedness of $T^{+}$ by a standard interpolation argument \cite{OST}. Therefore an interesting problem will be formulated as

\medskip
\noindent\textbf{Question 2}\quad Does $T^{+}$ map $H_{+}^{1}(w)$ into $L^{1}(w)$ with $w\in A_{1}^{+}$?
\medskip

However, the above problem is still open even in the classical ``two-sided" case (see \cite{L2}). In the present note, we partly answer this question when $P(x,y)=P(x-y)$. In this case, $T^{+}$ is a convolution operator, which can allow us to use Fourier transform. In fact, we can prove the following results.

\begin{theorem} \label{Theorem 1.1}
Let $P(x)$ be a polynomial which satisfies $P'(0)=0$ and $w\in A_{1}^{+}$. Then there exists a constant $C>0$, which depends only on $A_{1}^{+}(w)$
and the degree of $P(x)$ (not its coefficients), such that\\
$$\|T^{+}f\|_{L^{1}(w)}\leq C\|f\|_{H^{1}_{+}(w)}$$
for all $f\in H^{1}_{+}(w)$.
\end{theorem}
Rosa and Segovia \cite{RS1} also considered $[H^{q}_{+}(w)]^{\ast}$-the dual space of $H^{q}_{+}(w)$ formed by all the real valued continuous linear functions $F$ with the norm
$$
\|F\|=\sup\{|F(f)|:\|f\|_{H^{q}_{+}(w)}\leq 1\}.
$$
They gave a characterization of $[H^{q}_{+}(w)]^{\ast}$ in terms of certain classes one-sided weighted $BMO$ of Lipschitz spaces. We will touch only a few aspects of this theory and refer to \cite{RS1} for more details. We can obtain
\begin{corollary} \label{corollary 1.1}
Let $P$ and $w$ be in Theorem \ref{Theorem 1.1}. Then $T^{+}$ is bounded from $L^{\infty}(w)$ into $[H^{q}_{+}(w)]^{\ast}$.
\end{corollary}

We would like to point out that the restriction $P'(0)=0$ in Theorem \ref{Theorem 1.1} is essential. For example, we take $w=1$, $P(x)=\lambda x(\lambda> 0)$, $f(x)=\pi^{-1}\lambda\left(\chi_{[0,\pi/2\lambda]}(x)-\chi_{[-\pi/2\lambda,0]}(x)\right)$ and $K(x)=\frac{\sin(\log|x|)}{(x\log|x|)}\chi_{(-\infty,0)}(x)$.
Let
$$
T^{+}f(x)=\mathrm{p.v.}\int_{x}^{\infty}e^{iP(x-y)}K(x-y)f(y)dy.
$$
Below we will let $x<-100\lambda^{-1}$. Let
$$
g(x)=T^{+}f(x)-e^{i\lambda x}K(x)\int_{x}^{\infty}e^{-i\lambda y}f(y)dy.
$$
Then
$$|g(x)|\leq \int_{x}^{\infty}|K(x-y)-K(x)||f(y)|dy\leq C\lambda/|x|^{2},$$
which implies
\begin{align*}
|T^{+}f(x)|&\geq \left|e^{i\lambda x}K(x)\int_{x}^{\infty}e^{-i\lambda y}f(y)dy\right|-|g(x)|\\
&\geq|K(x)|\left|2\lambda/\pi\int_{0}^{\pi/2\lambda}\sin(\lambda y)dy\right|-C\lambda^{-1}|x|^{-2}\\
&= (2/\pi)|K(x)|-C\lambda^{-1}|x|^{-2}.
\end{align*}
Therefore,
$$
\int_{\mathbb{R}}|T^{+}f(x)|dx\geq (2/\pi)\int_{-\infty}^{-100\lambda^{-1}}|K(x)|dx-C/100=\infty.
$$

It is well known that for $q<1$, $\widetilde{T}^{+}$ is still bounded from $H^{q}_{+}(w)$ to $H^{q}_{+}(w)$ (see \cite[Theorem 3.1]{OS}). However, this is no longer suitable for $T^{+}$. In Section \ref{section3} of this article, we will show that this fails even $P(x,y)$ is the bilinear phase function by following a simple counterexample adopting from \cite{P}.

When $1<p<\infty$, there is still an interesting question for general $P(x,y)$:

\medskip
\noindent\textbf{Question 3}\quad Is $T^{+}$ bounded on $L^{p}(w) (1<p<\infty)$ with $w\in A_{p}^{+}$?
\medskip

Recently, we proved the weighted $L^{p}(1<p<\infty)$ estimates for $T^{+}$ in \cite{FSL} and showed that for any real polynomial $P(x,y)$, $T^{+}$ is of type $(L^{p}(w), L^{p}(w))$ for $w\in A^{+}_{p}$. It is easy to see that when $P(x,y)$ is trivial, for example, $P(x,y)=0$, then $T^{+}$ is
$\widetilde{T}^{+}$.
In \cite{AFM}, the authors proved that $\widetilde{T}^{+}$ enjoys weighted $L^{p}(1<p<\infty)$ boundedness properties similar to those of $M^{+}$. As a result of this close relationship between $T^{+}$ and $\widetilde{T}^{+}$, there is one meaningful problem: Are there some connections between the boundedness of these two one-sided operators? In this paper, we shall give a criterion for the weighted $L^{p}$-boundednesss of $T^{+}$ and show that the boundedness of  $T^{+}$ can be deduced by the corresponding boundedness of $\widetilde{T}^{+}$.

\begin{theorem} \label{Theorem 1.2}
 Let $P(x,y)$ be a real polynomial, $K$ be a one-sided Calder\'{o}n-Zygmund kernel and $b(r)$ be a bounded variation function on $[0,\infty)$. For $1<p<\infty$ and $w\in A^{+}_{p}$, we have

$\mathrm{(a)}$\,\, The operator
$$
\widetilde{T}^{+,b}f(x)=p.v \int_{x}^{\infty}b(y-x)K(x-y)f(y)dy
$$
is of type $(L^{p}(w), L^{p}(w))$.

$\mathrm{(b)}$\,\, The operator
$$
T^{+,b}f(x)=p.v \int_{x}^{\infty}e^{iP(x,y)}b(y-x)K(x-y)f(y)dy
$$ is of type $(L^{p}(w), L^{p}(w))$ .
Here its norm depend only on the total degree of $P(x,y)$ and $A_{p}^{+}(w)$, but not on the coefficients of $P(x,y)$.
\end{theorem}

Furthermore, we have
\begin{theorem}  \label{Theorem 1.3}
Let $w$, $p$ and $K$ be as in Theorem \ref{Theorem 1.2}. Then the following statements are equivalent:

$\mathrm{(a)}$\,\, If $P(x,y)$ is a nontrivial polynomial $\mathrm{(}$$P(x,y)$ does not take the form $P_{0}(x)+P_{1}(y)$, where $P_{0}$ and $P_{1}$ are polynomials defined on $\mathbb{R}$ $\mathrm{)}$, then the operator $T^{+}$ is of type $(L^{p}(w), L^{p}(w))$.

$\mathrm{(b)}$\,\, If $P(x,y)$ satisfies $P(x,y)=P(x-h,y-h)+P_{0}(x,h)+P_{1}(y,h)$ with $h\in \mathbb{R}$ and $P_{0}$ and $P_{1}$ are polynomials defined on $\mathbb{R}$, then the operator $T^{+}$ is of type $(L^{p}(w), L^{p}(w))$.

$\mathrm{(c)}$\,\, The truncated operator
$$
\widetilde{T}_{0}^{+}f(x)=p.v \int_{x}^{x+1}K(x-y)f(y)dy
$$
is of type $(L^{p}(w), L^{p}(w))$.
\end{theorem}

Theorem \ref{Theorem 1.2} and Theorem \ref{Theorem 1.3} readily produces the following result for the maximal operator corresponding to $T^{+}$:
\begin{theorem} \label{Theorem 1.4}
Let $w$, $p$ and $K$ be as in Theorem \ref{Theorem 1.2}. Then the maximal operator
$$
T_{\ast}^{+}f(x)=\sup_{\varepsilon>0} \left|\int_{x+\varepsilon}^{\infty}e^{iP(x,y)}K(x-y)f(y)dy\right|
$$
is of type $(L^{p}(w), L^{p}(w))$, where its norm depends only on the total degree of $P(x,y)$, but not on the coefficients of $P(x,y)$.
\end{theorem}

We end this section with the outline of this paper. Section \ref{section2} contains the proof of Theorem \ref{Theorem
1.1} an a counterexample to show that the boundedness for $T^{+}$ in Theorem \ref{Theorem 1.1} can not be extended to $H^{q}_{+}(w)$ when $q<1$ and $w\in A_{p}^{+}(1<p<\infty)$.  In Section \ref{section3}, the proofs of Theorem \ref{Theorem 1.2}-Theorem \ref{Theorem 1.4} will be given.

\section{One-sided estimates on weighted Hardy spaces}\label{section2}

In order to prove Theorem \ref{Theorem
1.1}, we first collect some lemmas. If $w(x)\in A_{p}$, then it is a doubling weight, that is, there exists $C>0$ such that
$$
\int_{a-2h}^{a+2h}w\leq C\int_{a-h}^{a+h}w
$$
for all $a\in \mathbb{R}$ and $h>0$. However, one-sided weights $A_{p}^{+}$ do not satisfy this property. But the weights $A_{p}^{+}$ satisfy a one-sided doubling condition:

\begin{lemma} \cite{RiT}\,\,\label{lemma 2.2} Let $w(x)\in A_{p}^{+}(p\geq 1)$. Then there exists a constant $C>0$ such that
$$
\int_{a-h}^{a+h}w\leq C\int_{a}^{a+h}w
$$
for all $a\in \mathbb{R}$ and $h>0$.
\end{lemma}
For $\lambda>1$, we denote by $I^{-}=[a-h,a]$ and $\lambda I^{-}=[a-\lambda h,a]$. Therefore, for $w\in A_{1}^{+}$, we have
\begin{equation}\label{2.2}
w(\lambda I^{-})\leq C\lambda w(I^{-})
\end{equation}
by Lemma \ref{lemma 2.2}, see also \cite[Proposition 12]{SF}.

Besides the doubling condition, $A_{p}$ weights satisfy the reverse H\"{o}lder inequality which play a key role in the proof of the strong type $(p,p)$ inequality of operators. If $w\in A_{p}$, then there exists $\delta>0$ and $C>0$ such that
$$
\frac{1}{b-a}\int_{a}^{b}w^{1+\delta}\leq \left(\frac{1}{b-a}\int_{a}^{b}w\right)^{1+\delta}
$$
for all intervals $(a,b).$ Unfortunately, one-sided weights $A_{p}^{+}$ do not satisfy the reverse H\"{o}lder inequality. However, a substitute was found in \cite{Ma}:
\begin{lemma} \,\,\label{lemma 2.3}
If $w\in A_{p}^{+}(p\geq1)$, then there exists constants $C$ and $\delta$ such that for all $a$ and $b$
\begin{equation}\label{2.3}
\int_{a}^{b}w^{1+\delta}\leq C\left(M^{-}(w\chi_{(a,b)})(b)\right)^{\delta}\int_{a}^{b}w.
\end{equation}
\end{lemma}
(\ref{2.3}) implies that
$$
M^{-}\left(w^{1+\delta}\chi_{(a,b)}\right)(b)\leq C\left(M^{-}(w\chi_{(a,b)})(b)\right)^{1+\delta},
$$
which is what we have called the weak reverse H\"{o}lder inequality since
$$
\left(M^{-}(w\chi_{(a,b)}(b)\right)^{1+\delta}\leq M^{-}\left(w^{1+\delta}\chi_{(a,b)}\right)(b)
$$
by the H\"{o}lder inequality.

We point out that it was proved in \cite{Ma} that (\ref{2.3}) holds if and only if there exists positive numbers $\delta$ and $C$ such that
\begin{equation}\label{2.4}
\frac{1}{c-a}\int_{a}^{c}w^{1+\delta}\leq C\left(\frac{1}{b-a}\int_{a}^{b}w\right)^{1+\delta}
\end{equation}
for all numbers $a<b$ and $c=(a+b)/2$, which seems to be a more natural formulation.

Let $w\in A_{1}^{+}$, $I^{-}=[x_{0}-h,x_{0}]$, and $a(x)$ be a $H^{1}_{+}(w)$ atom, which satisfies

(a)\,\,$\supp(a)\subset I^{-}$;

(b)\,\,$\int_{I^{-}}a(x)dx=0;$

(c)\,\,$\|a\|_{L^{\infty}}\leq w(I^{-})^{-1}$.\\
Let $I_{0}^{-}=[-1,0]$. Then we have $w_{0}(x)=w(x_{0}+hx)$. It is easy to see that $w_{0}\in A_{1}^{+}$ and $A_{p}^{+}(w_{0})=A_{p}^{+}(w)$. Set
$b(x)=ha(x_{0}+hx)$, we see that $b(x)$ is a $H^{1}_{+}(w_{0})$ atom, and it satisfies
\begin{equation}\label{2.5}
\supp (b)\subset I_{0}^{-},
\end{equation}
\begin{equation}\label{2.6}
\int_{I_{0}^{-}}b(x)dx=0,
\end{equation}
\begin{equation}\label{2.7}
\|b\|_{L^{\infty}}\leq w_{0}(I_{0}^{-})^{-1}.
\end{equation}
Furthermore, we have
$$
(T^{+}a)(x_{0}+hx)=h^{-1}(T_{1}^{+}b)(x),
$$
where $T_{1}^{+}f(x)=\int_{I_{0}^{-}}e^{iP(hx-hy)}K(x-y)b(y)dy$, which leads to $\|T^{+}a\|_{L^{1}(w)}=\|T_{1}^{+}b\|_{L^{1}(w)}.$

To prove Theorem \ref{Theorem 1.1}, we first prove the following proposition.

\begin{proposition}\label{proposition 2.1}
Let $P(x)$ be a polynomial with $P'(0)=0$ and $w\in A_{1}^{+}$. Then for any $H_{1}^{+}(w)$ atom $a(x)$, we have
$$
\|T^{+}(a)\|_{L^{1}(w)}\leq C,
$$
where $C$ is a constant, depending only on the degree of $P(x)$ and $A_{1}^{+}(w)$.
\end{proposition}
The preceding argument shows that it is sufficient to prove Proposition \ref{proposition 2.1} for $H_{1}^{+}(w)$ atoms which satisfy (\ref{2.5})-(\ref{2.7}) (with $w(x)$ replaced by $w_{0}(x)$). We first list a few lemmas that are needed in the proof of Proposition \ref{proposition 2.1}.

\begin{lemma}\,\,\cite{RS}\label{lemma 2.4}
Let $Q(x)=\sum_{\alpha\leq d}q_{\alpha}x^{\alpha}$ be a polynomial in $x\in\mathbb{R}$, with degree $d$. Suppose $\varepsilon<1/d$. Then
$$
\int_{|x|\leq 1}|Q(x)|^{-\varepsilon}dx\leq A_{\varepsilon}\left(\sum_{\alpha\leq d}|q_{\alpha}|\right)^{-\varepsilon}.
$$
\end{lemma}

\begin{lemma}\,\,\cite{RS}\label{lemma 2.5}
Let $\psi\in C^{1}[\alpha,\beta]$, $\varepsilon=\min\{{1}/{a_{1}},{1}/{n}\}$, $\lambda>0$. Then
$$
\left|\int_{\alpha}^{\beta}e^{i\lambda\phi(t)}\psi(t)dt\right|\leq C\lambda^{-\varepsilon}\left\{\sup_{\alpha\leq t\leq\beta}|\psi(t)|+\int_{\alpha}^{\beta}|\psi'(t)|dt\right\},
$$
where $\phi$ is real-valued phase of the form $\phi(t)=t^{a_{1}}+\mu_{2}t^{a_{2}}+\cdots+\mu_{n}t^{a_{n}}$
with real parameters $\mu_{2},\cdots,\mu_{n}$  and distinct positive exponents $a_{1},a_{2},\cdots,a_{n}$.
\end{lemma}
\begin{lemma}\,\,\label{lemma 2.6}
Let $P(x)$ be a polynomial of degree $m(m\geq 2)$ and $P(x)=\sum_{\alpha\leq m}a_{\alpha}x^{\alpha}$. Suppose $\varphi$ and $\psi$ are two functions in $C_{0}^{\infty}$. Define $T_{j}^{+}$ by
$$
(T_{j}^{+}f)(x)=\psi(2^{-j}x)\int_{x}^{\infty}e^{iP(x-y)}\varphi(y)f(y)dy.
$$
Then we have
$$
\|T_{j}^{+}f\|_{L^{2}(\mathbb{R})}\leq C|a_{m}|^{-1/(4(m-1))}2^{j/4}\|f\|_{L^{2}(\mathbb{R})}.
$$
\end{lemma}
Combining Lemma \ref{lemma 2.4} with Lemma \ref{lemma 2.5}, we can prove Lemma \ref{lemma 2.6} by a similar analysis as in \cite{HP}, corresponding argument, see also \cite{FLSS} and \cite{RS}.

The following proposition about one-sided Calder\'{o}-Zygmund $\widetilde{T}^{+}$ play a key role in the proof of Proposition \ref{proposition 2.1}.
\begin{proposition}\label{proposition 2.2}
Let $\widetilde{T}^{+}$ be a one-sided Calder\'{o}-Zygmund operator and $a(x)$ be a $H^{1}_{+}(w)$ atom satisfy $(\ref{2.5})$-$(\ref{2.7})$. Then
$$
\|\widetilde{T}^{+}a\|_{L^{1}(w)}\leq C.
$$
\end{proposition}
\begin{proof}
Let $\supp a\subset I_{0}^{-}=[-1,0]$ and $\widetilde{I}^{-}=2I_{0}^{-}=[-2,0]$. Then for $x\in \widetilde{I}^{-,c} \equiv(x_{-\infty},\infty)\setminus \widetilde{I}^{-}$, we have
\begin{align*}
|\widetilde{T}^{+}a(x)|&\leq \int_{x}^{\infty}|K(x-y)-K(x)||a(y)|dy\\
&\leq
\frac{1}{|x|^{2}}\int_{I_{0}^{-}}|y||a(y)|dy\\
&\leq
\frac{1}{|x|^{2}}w(I_{0}^{-})^{-1}.
\end{align*}
This implies
\begin{align*}
\int_{\widetilde{I}^{-,c}}|\widetilde{T}^{+}a(x)|w(x)dx
&\leq
C\int_{\widetilde{I}^{-,c}}\frac{1}{|x|^{2}}w(I_{0}^{-})^{-1}w(x)dx\\
&\leq
Cw(I_{0}^{-})^{-1}\int_{x_{-\infty}}^{-2}\frac{w(x)}{|x|^{2}}dx\\
&\leq
Cw(I_{0}^{-})^{-1}\sum_{j=1}^{\infty}2^{-2j}\int_{-2^{j+1}}^{0}w(x)dx\\
&\leq
C\sum_{j=1}^{\infty}2^{-j}\leq
C,
\end{align*}
where we use (\ref{2.2}) and (\ref{2.7}).

For $x\in \widetilde{I}^{-}$, it is easy to get
\begin{align*}
\int_{\widetilde{I}^{-}}|\widetilde{T}^{+}a(x)|w(x)dx
&\leq
(\int_{\mathbb{R}}|T^{+}a|^{2}w(x)dx)^{1/2}(\int_{-2}^{0}w(x)dx)^{1/2}\\
&\leq
C(\int_{\mathbb{R}}|a|^{2}w(x)dx)^{1/2}w(\widetilde{I}^{-})^{1/2}\\
&\leq
C\|a\|_{L^{\infty}}w(I_{0}^{-})
\leq
C.
\end{align*}
\end{proof}

We now come back to the proof of Proposition \ref{proposition 2.1}. Assume that $a$ is a $H^{1}_{+}(w)$ atom that satisfies (\ref{2.5})-(\ref{2.7}). Adopting the idea in \cite{HP}, we shall prove Proposition \ref{proposition 2.1} by using induction on $m$, the degree of $P(x)$. When $m=0$, that is
$P(x)=0$, which imply $T^{+}=\widetilde{T}^{+}$ in that case. Proposition \ref{proposition 2.1} holds by Proposition \ref{proposition 2.2}. We now assume that Proposition \ref{proposition 2.1} is true for $deg(P)\leq m-1$. The task is now to show Proposition \ref{proposition 2.1} for $deg(P)= m$. We write
$$
P(x-y)=a_{m}(x-y)^{m}+P_{m-1}(x-y),
$$
where $deg(P_{m-1})\leq m-1$. Let $b=\max\{|a_{m}|^{-1/(m-1)},2\}$. We distinguish two cases to obtain our desired results.

{\it Case} 1. $b<-x_{-\infty}$. In this case, we break the integral into three parts:
\begin{align*}
\|T^{+}a\|_{L^{1}(w)}&\leq \left|\int_{|x|\leq b}T^{+}a(x)w(x)dx\right|+\left|\int_{x_{-\infty}}^{-b}T^{+}a(x)w(x)dx\right|+\left|\int_{b}^{\infty}T^{+}a(x)w(x)dx\right|\\
&=:I_{1}+I_{2}+I_{3}.
\end{align*}

The first step is to show that $I_{1}\leq C.$ If $b=2,$ the estimates follows from a standard argument as
\begin{align*}
I_{1}&= \left|\int_{|x|\leq 2}T^{+}a(x)|w(x)dx\right|\\
&\leq \|T^{+}(a)\|_{L^{2}(w)}\left(\int_{-2}^{0}w(x)dx\right)^{1/2}\\
&\leq C\|a\|_{L^{2}(w)}w(I_{0}^{-})^{1/2}\\
&\leq C\|a\|_{L^{\infty}}w(I_{0}^{-})\\
&\leq C,
\end{align*}
where we use (\ref{2.2}), (\ref{2.7}) and the weighted $L^{p}$ estimate for $T^{+}$ (\cite{FSL}).

Assuming $b=|a_{m}|^{-1/(m-1)}$, by the above argument for $b=2$, we have
\begin{align*}
I_{1}&\leq \left|\int_{|x|\leq 2}T^{+}a(x)|w(x)dx\right|+\left|\int_{2\leq|x|\leq b}T^{+}a(x)|w(x)dx\right|\\
&\leq C+\left|\int_{- b}^{-2}\left|\int_{x_{-\infty}}^{\infty}e^{iP_{m-1}(x-y)}K(x-y)a(y)dy\right|w(x)dx\right|\\
&\quad +
\left|\int_{- b}^{-2}\left|\int_{x_{-\infty}}^{\infty}e^{i(P(x-y)-P_{m-1}(x-y)-a_{m}x^{m})}K(x-y)a(y)dy\right|w(x)dx\right|\\
&=:C+J_{1}+J_{2}.
\end{align*}

By inductive hypothesis, $J_{1}\leq C.$ On the other hand, we have
\begin{align*}
J_{2}&\leq C|a_{m}|\int_{-b}^{-2}\int_{I_{0}^{-}}\frac{|(x-y)^{m}-x^{m}|}{y-x}|a(y)|dyw(x)dx\\
&\leq C|a_{m}|\int_{- b}^{-2}|x|^{m-2}w(x)dx\int_{I_{0}^{-}}|a(y)|dy\\
&\leq C|a_{m}|\sum_{j\geq 1, 2^{j}\leq b}2^{j(m-2)}\int_{- b}^{0}w(x)dx\int_{I_{0}^{-}}|a(y)|dy\\
&\leq C|a_{m}|b^{m-1}\\
&\leq C.
\end{align*}

Next, we prove that $I_{2}\leq C$. Assume that $2^{j_{0}}\leq b\leq 2^{j_{0}+1}$. Let $\varphi\in C_{0}^{\infty}(\mathbb{R})$ and $\varphi=1$ on $I_{0}^{-}$. Choosing $\psi\in C_{0}^{\infty}(\mathbb{R})$ such that $\supp (\psi)\subset \{1/4<|x|<4\}$, $\psi\geq 0$ and $\psi(x)=1$, for $1\leq |x|\leq 2.$ We have
\begin{align*}
I_{2}&\leq \int_{x_{-\infty}}^{-b}\int_{x}^{\infty}|K(x-y)-K(x)||a(y)|dyw(x)dx+\int_{x_{-\infty}}^{-b}\left|\int_{x}^{\infty}e^{iP(x-y)}a(y)dyK(x)w(x)\right|dx\\
&=: K_{1}+K_{2}.
\end{align*}

Since $K$ is one-sided Calder\'{o}-Zygmund kernel, by (\ref{2.7}), we can estimate $K_{1}$ as
\begin{align*}
K_{1}&\leq \int_{x_{-\infty}}^{-2}\|a\|_{L^{\infty}}\int_{-1}^{0}dy\frac{w(x)}{|x|^{2}}dx\\
&\leq \frac{C}{w(I_{0}^{-})}\int_{x_{-\infty}}^{-2}\frac{w(x)}{|x|^{2}}dx\\
&\leq \frac{C}{w(I_{0}^{-})}\sum_{j=1}^{\infty}2^{-2j}w(2^{j+1}I_{0}^{-})\\
&\leq C.
\end{align*}
While the H\"{o}lder inequality allows us to estimate $K_{2}$ as
\begin{align*}
K_{2}&\leq \sum_{j\geq j_{0}}\int_{-2^{j+1}}^{-2^{j}}\frac{w(x)}{|x|}\psi(2^{-j}x)\left|\int_{x}^{\infty}e^{iP(x-y)}\varphi(y)a(y)dy\right|dx\\
&= \sum_{j\geq j_{0}}\int_{-2^{j+1}}^{-2^{j}}\frac{w(x)}{|x|}|T_{j}^{+}(a)|dx\\
&\leq \sum_{j\geq j_{0}}\|T_{j}^{+}(a)\|_{L^{p}}\left(\int_{-2^{j+1}}^{-2^{j}}\frac{w(x)^{1+\varepsilon}}{|x|^{1+\varepsilon}}dx\right)^{1/(1+\varepsilon)}.
\end{align*}

Invoking the properties (a) of $A_{1}^{+}$ weights and (\ref{2.4}), we obtain
$$
\int_{-2^{j+1}}^{-2^{j}}\frac{w(x)^{1+\varepsilon}}{|x|^{1+\varepsilon}}dx
\leq C2^{-j\varepsilon}\frac{2^{j+1}w^{1+\varepsilon}(I_{0}^{-})}{|2^{j+1}I_{0}^{-}|}
\leq C2^{-j\varepsilon} w^{1+\varepsilon}(I_{0}^{-}).
$$
Therefore,
$$
I_{2}\leq C+\sum_{j\geq j_{0}}\|T_{j}^{+}(a)\|_{L^{p}}2^{-j\varepsilon/(1+\varepsilon)}w(I_{0}^{-}).
$$

After noting that
$$
\|T_{j}^{+}(a)\|_{L^{\infty}}\leq C\|a\|_{L^{\infty}},
$$
by Lemma \ref{lemma 2.6} and interpolation we get
\begin{align*}
\|T_{j}^{+}(a)\|_{L^{p}}&\leq 2^{j/p}(|a_{m}|2^{j(m-1)})^{-1/(2p(m-1))}\left(\int_{I_{0}^{-}}|a|^{p}\right)^{1/p}\\
&\leq 2^{j/2p}(|a_{m}|)^{-1/(2p(m-1))}\|a\|_{L^{\infty}}\\
&\leq 2^{j/2p}(|a_{m}|)^{-1/(2p(m-1))}w(I_{0}^{-})^{-1}.
\end{align*}

The estimate for $I_{2}$ is completed by showing that
\begin{align*}
I_{2}&\leq C+\sum_{j\geq j_{0}}2^{j/2p}|a_{m}|^{-1/2p(m-1)}w(I_{0}^{-})^{-1}2^{-j\varepsilon/(1+\varepsilon)}w(I_{0}^{-})\\
&\leq C+|a_{m}|^{-1/2p(m-1)}2^{-j_{0}/2p}\\
&\leq C+|a_{m}|^{-1/2p(m-1)}b^{-1/2p}\\
&\leq C.
\end{align*}

On the other hand, $\supp K=(-\infty,0)$ and $\supp a\subset I_{0}^{-}$ show that
$$
I_{3}=0.
$$
We conclude from above estimate for $I_{1}$, $I_{2}$ and $I_{3}$ that
$$
\|T^{+}(a)\|_{L^{1}(w)}\leq C.
$$

{\it Case} 2. $b>-x_{-\infty}$. In this case, we have
\begin{align*}
\|T^{+}(a)\|_{L^{1}(w)}&\leq \left|\int_{x_{-\infty}}^{b}T^{+}(a)(x)w(x)dx\right|+\left|\int_{b}^{\infty}T^{+}(a)(x)w(x)dx\right|\\
&=: \widetilde{I}_{1}+\widetilde{I}_{2}.
\end{align*}

Similar as in the estimate of $I_{3}$, we have $\widetilde{I}_{2}=0$. So, we only need to consider $\widetilde{I}_{1}$. If $b=2,$ by Lemma \ref{lemma 2.2} and (\ref{2.7}), we have
\begin{align*}
\widetilde{I}_{1}&\leq \|T^{+}(a)\|_{L^{2}(w)}(\int_{x_{-\infty}}^{0}w(x)dx)^{1/2}\\
&\leq \|a\|_{L^{2}(w)}\left(\int_{-1}^{0}w(x)dx\right)^{1/2}\\
&\leq \|a\|_{L^{\infty}}w(I_{0}^{-})\\
&\leq C.
\end{align*}

If $b=|a_{m}|^{-1/(m-1)}$. Applying the above estimates for $\widetilde{I}_{1}$ when $b=2,$ we have
\begin{align*}
\widetilde{I}_{1}&\leq \left|\int_{x_{-\infty}}^{2}T^{+}(a)w(x)dx\right|+\left|\int_{2}^{b}T^{+}(a)w(x)dx\right|\\
&\leq C+\widetilde{J}_{1}\\
&\leq C
\end{align*}
as a result of $\widetilde{J}_{1}=0.$

Combining {\it Case} 1 and {\it Case} 2, we have thus proved Proposition \ref{proposition 2.1}.

Having disposed of the above preliminaries, we can now return to the proof of our main theorem.
As a byproduct of Proposition \ref{proposition 2.1}, we have
$$
\|T^{+}f\|_{L^{1}(w)}\leq \sum_{j=1}^{\infty}|\lambda_{j}|\|T^{+}(a_{j})\|_{L^{1}(w)})\leq C\sum_{j=1}^{\infty}|\lambda_{j}|\leq C\|f\|_{H_{+}^{1}(w)}.
$$

We have completed the proof of Theorem \ref{Theorem 1.1}.

Inspired by the main idea from \cite{P}, in this section, a counterexample is given to show that the $H_{+}^{1}(w)$ theory on the one-sided oscillatory singular operators can not be extended to the $H_{+}^{q}(w)$ case, if $q<1$. Let $\overline{T}^{+}$ be defined as
$$
\overline{T}^{+}f(x)=p.v.\int_{x_{-\infty}}^{\infty}e^{ixy}\frac{f(y)}{x-y}dy.
$$
Take $\delta>0$ very small, and $\supp f\subset I_{\delta}=[-\delta,\delta]$ given by
$$
f(y)=\left\{\begin{array}{ll}&(2\delta)^{-1/q},\,\,\,\,\,\,y\in [\delta/2,\delta];\\
&(-2\delta)^{-1/q},\,\,\,\,\,\,y\in [-\delta,-\delta/2];\\
&0,\,\,\,\,\,\,\mathrm{otherwise}.\end{array}\right.
$$

It is easy to check that $|f|\leq |I_{\delta}|^{-1/q}$, $\int_{I_{\delta}}f(y)dy=0.$ Therefore, we have
$$
\mathrm{Im}(T^{+}a)(x)=(2\delta)^{-1/q}\left(\int_{\delta/2}^{\delta}\frac{\sin(xy)}{x-y}dy+\int_{\delta/2}^{\delta}\frac{\sin(xy)}{x+y}dy\right).
$$

Let $x\in (-\pi/3\delta, -\pi/4\delta)$. Then $x-y<0$, $x+y<0$ for any $y\in [\delta/2,\delta]$. Also, we have $-\pi/3<xy<-\pi/8$. Thus
\begin{align*}
|\mathrm{Im}(T^{+}a)(x)|&>C(2\delta)^{-1/q}\left(\int_{\delta/2}^{\delta}\frac{dy}{y-x}+\int_{\delta/2}^{\delta}\frac{dy}{x-y}\right)\\
&=C\delta^{-1/q}\ln(1+\frac{-\delta x}{x^{2}+\delta/2 x-\delta^{2}/2})\geq C\delta^{-1/q}(-x)^{-1}.
\end{align*}
Therefore, we have
$$
\int_{\mathbb{R}}|T^{+}a(x)|^{q}dx\geq \int_{-\pi/3\delta}^{-\pi/4\delta}\delta^{q-1}(-x)^{-q}dx=\delta^{2(q-1)}.
$$

Set $\max\{-1,2(q-1)\}<\alpha\leq 0$. Then $w=|x|^{\alpha}\in A_{1}^{+}$ and
$$\int_{\mathbb{R}}|T^{+}a(x)|^{q}w(x)dx\geq C\delta^{q-1}\int_{\pi/4\delta}^{\pi/3\delta}t^{\alpha-q}dt\geq C\delta^{2(q-1)-\alpha}\rightarrow\infty$$
by letting $\delta\rightarrow 0$ since $q<1$.

\section{Criterion on weighted $L^{p}$ estimates}\label{section3}
In this section, a criterion on boundedness of the one-sided operators mentioned in Section \ref{section1} and its effects on weighted $L^{p}$ spaces are described. Let us first begins with some properties about the $A_{p}^{+}$ classes, which will be used in the proof of our main results.

\begin{lemma}\label{Lemma4.1}\cite{FSL}
Let $1<p<\infty$ and $w\in A_{p}^{+}$. Then\\
$(a)$\,\,$ A_{p}^{+}(\delta^{\lambda}(w))=A_{p}^{+}(w)$, where $\delta^{\lambda}(w)(x)=w(\lambda x)$ for all $\lambda>0$.\\
$(b)$\,\,there exists $\varepsilon>0$ such that $w^{1+\varepsilon}\in A_{p}^{+}$.
\end{lemma}
The following celebrated interpolation theorem of operators with change of measures will be needed in our analysis.
\begin{lemma}\label{Lemma4.2}\cite{SW}
Suppose that $u_{0},v_{0},u_{1},v_{1}$ are positive weight functions and $1<p_{0},p_{1}<\infty$. Assume sublinear operator $S$ satisfies:
$$\|Sf\|_{L^{p_{0}}(u_{0})}\leq C_{0}\|f\|_{L^{p_{0}}(v_{0})},$$
and
$$\|Sf\|_{L^{p_{1}}(u_{1})}\leq C_{1}\|f\|_{L^{p_{1}}(v_{1})}.$$
Then
$$\|Sf\|_{L^{p}(u)}\leq C\|f\|_{L^{p}(v)}$$
holds for any $0<\theta<1$ and ${1}/{p}={\theta}/{p_{0}}+{(1-\theta)}/{p_{1}}$, where
$u=u_{0}^{{(p\theta)}/{p_{0}}}u_{1}^{{p(1-\theta)}/{p_{1}}}$, $v=v_{0}^{{(p\theta)}/{p_{0}}}v_{1}^{{p(1-\theta)}/{p_{1}}}$ and $C\leq C_{0}^{\theta}C_{1}^{1-\theta}$.
\end{lemma}
To prove Theorem \ref{Theorem 1.2}, the following lemma is still needed:
\begin{lemma}\label{Lemma4.3}
Suppose that $1<p<\infty$, $w\in A^{+}_{p}$ and $K$ satisfies $|K(x,y)|\leq {C}/{(y-x)}$. If the operator
$$\overline{T}^{+}=p.v\int_{x}^{\infty}K(x,y)f(y)dy$$ is of type $(L^{p}(w), L^{p}(w))$, then the operator
$$
\overline{T}_{\varepsilon}^{+}f(x)=\int_{x}^{x+\varepsilon}K(x,y)f(y)dy
$$
is of type $(L^{p}(w), L^{p}(w))$.
\end{lemma}
\begin{proof}
For $h\in \mathbb{R}$, decompose $f$ into three parts as
$$\begin{array}{rl}
\displaystyle f(y)&=\displaystyle f\chi_{\{|y-h|<{\varepsilon}/{2}\}}(y)+f\chi_{\{{\varepsilon}/{2}\leq|y-h|<{5\varepsilon}/{4}\}}(y)+f\chi_{\{|y-h|\geq{5\varepsilon}/{4}\}}(y)\\
&=:f_{1}(y)+f_{2}(y)+f_{3}(y).
\end{array}$$
When $|x-h|<{\varepsilon}/{4}$, it is easy to show
$\overline{T}_{\varepsilon}^{+}f_{1}(x)=\overline{T}^{+}f_{1}(x),$ which allows the following to be true
\begin{equation}\label{4.1}
\int_{|x-h|<{\varepsilon}/{4}}|\overline{T}_{\varepsilon}^{+}f_{1}(x)|^{p}w(x)dx\leq \int_{\mathbb{R}}|\overline{T}^{+}f_{1}(x)|^{p}w(x)dx\leq C\int_{|y-h|<{\varepsilon}/{2}} |f(y)|^{p}w(y)dy,
\end{equation}
where $C$ is independent of $h$ and the coefficients of $P(x,y)$.

The fact that if $|x-h|<{\varepsilon}/{4}$ and ${\varepsilon}/{2}\leq|y-h|<{5\varepsilon}/{4}$, then ${\varepsilon}/{4}<y-x<{3\varepsilon}/{2}$ allows the following to be shown
$$
|\overline{T}_{\varepsilon}^{+}f_{2}(x)|\leq C\int_{x+{\varepsilon}/{4}}^{x+\varepsilon}\frac{1}{(y-x)}|f_{2}(y)|dy\leq CM^{+}(f_{2})(x).
$$
On account of the boundedness of $M^{+}$, the following can be proved
\begin{equation}\label{4.2}
\int_{|x-h|<{\varepsilon}/{4}}|\overline{T}_{\varepsilon}^{+}f_{2}(x)|^{p}w(x)dx\leq C\int_{|y-h|<{5\varepsilon}/{4}} |f(y)|^{p}w(y)dy,
\end{equation}
where $C$ is independent of $h$ and the coefficients of $P(x,y)$.

Again notice that if $|x-h|<{\varepsilon}/{4}$ and $|y-h|\geq{5\varepsilon}/{4}$, then $y-x>\varepsilon$, the following can be shown
\begin{equation}\label{4.3}
\overline{T}_{\varepsilon}^{+}f_{3}(x)=0.
\end{equation}

Combining (\ref{4.1}), (\ref{4.2}) and (\ref{4.3}), the following
$$
\int_{|x-h|<{\varepsilon}/{4}}|\overline{T}_{\varepsilon}^{+}f(x)|^{p}w(x)dx\leq C\int_{|y-h|<{5\varepsilon}/{4}} |f(y)|^{p}w(y)dy.
$$
holds uniformly in $h\in \mathbb{R}$, where $C$ is independent of $h$ and the coefficients of $P(x,y)$, which implies
$$
\|\overline{T}_{\varepsilon}^{+}f\|_{L^{p}(w)}\leq C\|f\|_{L^{p}(w)},
$$
where $C$ is independent of the coefficients of $P(x,y)$.
\end{proof}

Having disposed of the above preliminary steps, the proofs of Theorem \ref{Theorem 1.2}-Theorem \ref{Theorem 1.4} can be addressed.

\subsection{Proof of Theorem \ref{Theorem 1.2}}\label{section3.1}
Using the same method in \cite{AFM}, the proof of (a) can be easily obtained. We omit it's proof here.

The proof of (b) is similar to that of
Theorem 1.5 in \cite{FSL}. This argument can now be applied again for the completeness of this paper. Suppose $P(x,y)$ is a real polynomial with degree $k$ in $x$ and degree $l$ in $y$. We shall carry out the argument by induction. For any nonzero real polynomial $P(x,y)$ in $x$ and $y$, there are $k, l, m\ge 0$ such that
\begin{equation}\label{4.4}
 P(x, y)
= a_{k l} x^k y^l + R(x,y)
\end{equation}
with $a_{kl} \ne 0$ and
\[ R(x,y)=\sum_{ 0 \le \alpha < k, 0 \le \beta \le m} a_{\alpha\beta} x^\alpha y^\beta + \sum_{0 \le \beta < l}   a_{k\beta} x^k y^\beta  .
\]
We shall write $d_x(P) = k$ and $d_y(P) = l$. Below we shall carry out the argument by using a double induction on $k$ and $l$.

If $d_x(P)=0$ and $d_y(P)$ is arbitrary,  then $P(x,y)=P(y)$ and $T^+f$ can be written as
\[T^+f(x)=\lim_{\varepsilon \to 0^+} \int^\infty_{x+\varepsilon} K(x-y) g(y) dy \]
where
$ g(y)=e^{iP(y)}f(y)$. Therefore, the conclusion of Theorem \ref{Theorem 1.2} follows from the assumption.

Let $k \ge 1$ and assume that the conclusion of Theorem \ref{Theorem 1.2} holds for all $P(x,y)$ with
$d_x(P) \le k-1$ and $d_y(P)$ arbitrary.

We will now prove that the conclusion of Theorem \ref{Theorem 1.2} holds for all $P(x,y)$ with
$d_x(P) = k$ and $d_y(P)$ arbitrary.

If $d_x(P)=k$ and $d_y(P)=0$, then
\[ P(x,y) = a_{k 0} x^k + Q(x,y) \]
with $d_x(Q) \le k-1$. By taking the factor $e^{i a_{k0} x^k}$ out of the integral sign, we see that this
case follows from the above inductive hypothesis.

Suppose $l \ge 1$ and the desired bound holds when $d_x(P) = k$ and $d_y(P) \le  l-1$. Now, let $P(x,y)$ be a polynomail with $d_x(P) = k$ and $d_y(P) = l$, as given in (\ref{4.4}).

{\it Case} 1. $|a_{kl}|=1.$

Write
$$\begin{array}{rl}
\displaystyle T^{+,b}f(x)&=\displaystyle \int_{x}^{1+x}e^{iP(x,y)}b(y-x)K(x-y)f(y)dy\\ &\quad\quad\displaystyle +\sum_{j=1}^{\infty}\int_{2^{j-1}+x}^{2^{j}+x}e^{iP(x,y)}b(y-x)K(x-y)f(y)dy\\
&=:\displaystyle T_{0}^{+,b}f(x)+T_{\infty}^{+,b}f(x)\\
&=:\displaystyle T_{0}^{+,b}f(x)+\sum_{j=1}^{\infty}T_{j}^{+,b}f(x).
\end{array}$$
Take any $h\in \mathbb{R}$, and write
$$
P(x,y)=a_{kl}(x-h)^{k}(y-h)^{l}+R(x,y,h),
$$
where the polynomial $R(x,y,h)$ satisfies the induction assumption, and the coefficients of $R(x,y,h)$ depend on $h$.

The estimates for $T_{0}^{+,b}$ is given first. It is easy to confirm that
$$\begin{array}{rl}
\displaystyle T_{0}^{+,b}f(x)&=\displaystyle \int_{x}^{1+x}e^{i(R(x,y,h)+a_{kl}(y-h)^{k+l})}b(y-x)K(x-y)f(y)dy\\
&\,\,+\displaystyle \int_{x}^{1+x}\left\{e^{iP(x,y)}-e^{i(R(x,y,h)+a_{kl}(y-h)^{k+l})}\right\}b(y-x)K(x-y)f(y)dy\\&=:\displaystyle T_{01}^{+,b}f(x)+T_{02}^{+,b}f(x).
\end{array}$$
Note that $\|b\|_{\infty}<+\infty$, by the induction assumption and Lemma \ref{Lemma4.3}, $T_{01}^{+,b}$ is a $(L^{p}(w), L^{p}(w))$ type operator
and the norm depends on $\|b\|_{\infty}$, but not on the the coefficients of $P(x,y)$ and $h$.

The estimate for the term $T_{02}^{+,b}$ can now be introduced. Evidently, if $|x-h|<{1}/{4}$ and $0<y-x<1$, then
$$
|e^{iP(x,y)}-e^{i(R(x,y,h)+a_{kl}(y-h)^{k+l})}|\leq |a_{kl}||x-y|=C(y-x).
$$
Therefore, when $|x-h|<{1}/{4}$, the following is true:
$$
|T_{02}^{+,b}f(x)|\leq C\|b\|_{\infty}\int_{x}^{x+1}|f(y)|dx\leq CM^{+}(f(\cdot)\chi_{B(h,\frac{5}{4})}(\cdot))(x).
$$
It follows that
$$
\int_{|x-h|<{1}/{4}}|T_{02}^{+,b}f(x)|^{p}w(x)dx\leq C\|b\|_{\infty}\int_{|y-h|<{5}/{4}} |f(y)|^{p}w(y)dy
$$
holds uniformly in $h\in \mathbb{R}$, which implies
\begin{equation}\label{4.5}
\|T_{0}^{+,b}f\|_{L^{p}(w)}\leq C\|f\|_{L^{p}(w)},
\end{equation}
where $C$ is independent of the coefficients of $P(x,y)$.

The estimates for $T_{\infty}^{+,b}f$ can now be given.
For $j\geq 1$, the following can be shown
$$
|T_{j}^{+,b}f(x)|\leq \int_{2^{j-1}+x}^{2^{j}+x}\frac{|f(y)|}{|x-y|}b(y-x)dy\leq C\|b\|_{\infty}M^{+}(f)(x),
$$
where $C$ is independent of $j$. By lemma \ref{Lemma4.1}, there exists $\varepsilon>0$, such that $w^{1+\varepsilon}\in A_{p}^{+}$. Thus
\begin{equation}\label{4.6}
\|T_{j}^{+,b}f\|_{L^{p}(w^{1+\varepsilon})}\leq C\|f\|_{L^{p}(w^{1+\varepsilon})},
\end{equation}
where $C$ is independent of $j$. On the other hand, recall Lemma 3.7 in \cite{FLSS} to see that
\begin{equation}\label{4.7}
\|T_{j}^{+,b}f\|_{L^{p}}\leq C2^{-j\delta}\|f\|_{L^{p}},
\end{equation}
where $C$ is dependents only on the total degree of $P(x,y)$ and $\delta>0$. From (\ref{4.6}) , (\ref{4.7}) and Lemma \ref{Lemma4.2}, it follows that
\begin{equation}\label{4.8}
\|T_{j}^{+,b}f\|_{L^{p}(w)}\leq C2^{-j\theta\delta}\|f\|_{L^{p}(w)},
\end{equation}
where $0<\theta<1$, $\theta$ is independent of $j$, and $C$ depends only on the total degree of $P(x,y)$.
Thus
\begin{equation}\label{4.9}
\|T_{\infty}^{+,b}f\|_{L^{p}(w)}\leq C\|f\|_{L^{p}(w)}.
\end{equation}

Now (\ref{4.5}) and (\ref{4.9}) imply that
\begin{equation}\label{4.10}
\|T_{b}^{+}f\|_{L^{p}(w)}\leq C\|f\|_{L^{p}(w)},
\end{equation}
where $C$ depends not on the the coefficients of $P(x,y)$.

{\it Case} 2. $|a_{kl}|\neq 1$.
Write $\lambda=|a_{kl}|^{{1}/{(k+l)}}$, and
$$
P(x,y)=\lambda^{-(k+l)}a_{kl}(\lambda x)^{k}(\lambda y)^{l}+R({\lambda x}/{\lambda},{\lambda y}/{\lambda})=Q(\lambda x,\lambda y).
$$
Thus
$$\begin{array}{rl}
\displaystyle T^{+,b}f(x)&=\displaystyle \mathrm{p.v.}\int e^{iQ(\lambda x, \lambda y)}b(y-x)K(x-y)f(y)dy\\
&=\displaystyle \mathrm{p.v.}\int e^{iQ(\lambda x,y)}b({(y-x)}/{\lambda})K(\lambda x-y)f({y}/{\lambda})dy
\end{array}$$

By Lemma \ref{Lemma4.1}, $\|b({(\cdot)}/{\lambda})\|_{\infty}=\|b\|_{\infty}$, so this case goes back to the result in case 1. On account of the estimates for case 1 and case 2 given above, the following can be proved:
$$
\|T^{+,b}f\|_{L^{p}(w)}\leq C\|f\|_{L^{p}(w)},
$$
where $C$ depends not on the coefficients of $P(x,y)$. \qed

\subsection{Proof of Theorem \ref{Theorem 1.3}}\label{section3.2}
(a) implies (b): This step is obvious.

(b) implies (c): Write
$$\begin{array}{rl}
\displaystyle T^{+}f(x)&=\displaystyle \int_{x}^{1+x}e^{iP(x,y)}K(x-y)f(y)dy+\int_{x+1}^{\infty}e^{iP(x,y)}K(x-y)f(y)dy\\
&=:\displaystyle T_{0}^{+}f(x)+T_{\infty}^{+}f(x).
\end{array}$$
From the method similar to the proof of (\ref{4.9}), $T_{\infty}^{+}$ is a $(L^{p}(w), L^{p}(w))$ type operator for $1<p<\infty, w\in A_{p}^{+}$, so is $T_{0}^{+}$.

Let $h\in \mathbb{R}$. Then for $|x-h|<1$,
$$T_{0}^{+}f(x)\leq T_{0}^{+}[f(\cdot)\chi_{I(h,2)}(\cdot)](x),$$
where $I(x_{0},r)$ denotes the interval $[x_{0}-r,x_{0}+r]$.
Thus,
\begin{align*}
\left(\int_{|x-h|<1} |T_{0}^{+}f(x)|^{p}w(x)dx\right)^{{1}/{p}}
&\leq\left(\int_{|x-h|<1} |T_{0}^{+}[f(\cdot)\chi _{I(h,2)}(\cdot)](x)|^{p}w(x)dx\right)^{{1}/{p}}\\
&=\|T_{0}^{+}f(\cdot)\chi _{I(h,2)}(\cdot)\|_{L^{p}(w)}\\
&\leq C\left(\int_{|y-h|<2} |f(y)|^{p}w(y)dy\right)^{{1}/{p}},
\end{align*}
where $C$ is independent of $h$.

Since $P(x,y)=P(x-h,y-h)+P_{0}(x,h)+P_{1}(y,h)$ with $h\in \mathbb{R}$ and $P_{0}$, $P_{1}$ are real polynomials defined on $\mathbb{R}$, it follows that
\begin{align*}
\widetilde{T}_{0}^{+}f(x)
&\leq p.v \int_{x}^{x+1}K(x-y)f(y)\chi _{I(h,2)}(y)dy\\
&=p.v e^{-iP_{0}(x,h)}\int_{x}^{x+1}e^{iP(x,y)}K(x-y)e^{-iP(x-h,y-h)}e^{-iP_{1}(y,h)}f(y)\chi _{I(h,2)}(y)dy.
\end{align*}
The Taylor's expression of $e^{-iP(x-h,y-h)}$ is
\begin{align*}
e^{-iP(x-h,y-h)}
&=\sum_{m=0}^{\infty}\frac{(-i)^{m}}{m!}\left[\sum_{\alpha,\beta}a_{\alpha,\beta}(x-h)^{\alpha}(y-h)^{\beta}\right]^{m}\\
&=\sum_{m=0}^{\infty}\frac{(-i)^{m}}{m!}\sum_{l}C_{m,l}b_{\alpha,\beta,l}(x-h)^{u(\alpha,\beta,l)}(y-h)^{v(\alpha,\beta,l)}.
\end{align*}
Therefore, if we set $|x-h|\leq A<1, |y-h|\leq B<2$, the following can be shown:
\begin{align*}
&\left(\int_{|x-h|<1} |\widetilde{T}_{0}^{+}f(x)|^{p}w(x)dx\right)^{{1}/{p}}\\
&=\Big(\int_{|x-h|<1}\big|e^{-iP_{0}(x,h)}\int_{x}^{x+1}e^{iP(x,y)}K(x-y)e^{-iP(x-h,y-h)}e^{-iP_{1}(y,h)}\\
&\,\,\,\,\,\,\,\,\,\,\,\,\,\,\,\times f(y)\chi _{I(h,2)}(y)dy\big|w(x)dx\Big)^{{1}/{p}}\\
&=\Big  (\int_{|x-h|<1}\big|e^{-iP_{0}(x,h)}\int_{x}^{x+1}e^{iP(x,y)}K(x-y)\sum_{m=0}^{\infty}\frac{(-i)^{m}}{m!}\sum_{l}C_{m,l}b_{\alpha,\beta,l}\\
&\,\,\,\,\,\,\,\,\,\,\,\,\,\,\,\times (x-h)^{u(\alpha,\beta,l)}(y-h)^{v(\alpha,\beta,l)}e^{-iP_{1}(y,h)}f(y)\chi _{I(h,2)}(y)dy\big|^{p}w(x)dx\Big)^{{1}/{p}}\\
&\leq \sum_{m=0}^{\infty}\sum_{l}\frac{|C_{m,l}b_{\alpha,\beta,l}|}{m!}\Big(\int_{|x-h|<1}\big|(x-h)^{u}\int_{x}^{x+1}e^{iP(x,y)}K(x-y)\\
&\,\,\,\,\,\,\,\,\,\,\,\,\,\,\,\times (y-h)^{v}e^{-iP_{1}(y,h)}f(y)\chi _{I(h,2)}(y)dy\big|^{p}w(x)dx\Big)^{{1}/{p}}
\end{align*}
\begin{align*}
&\leq \sum_{m=0}^{\infty}\sum_{l}\frac{|C_{m,l}b_{\alpha,\beta,l}|}{m!}A^{u}\Big(\int_{|x-h|<1}|T_{0}^{+}[e^{-iP_{1}(\cdot,h)}\\
&\,\,\,\,\,\,\,\,\,\,\,\,\,\,\,\times f(\cdot)\chi _{I(h,2)}(\cdot)(\cdot-h)^{v}](x)|^{p}w(x)dx\Big)^{{1}/{p}}\\
&\leq\sum_{m=0}^{\infty}\sum_{l}\frac{|C_{m,l}b_{\alpha,\beta,l}|}{m!}A^{u}\left(\int_{|y-h|<2}|f(y)|^{p}|(y-h)^{v}|^{p}w(y)dy\right)^{{1}/{p}}\\
&= \sum_{m=0}^{\infty}\sum_{l}\frac{|C_{m,l}b_{\alpha,\beta,l}|}{m!}A^{u}B^{v}\left(\int_{|y-h|<2}|f(y)|^{p}w(y)dy\right)^{{1}/{p}}\\
&=C\sum_{m=0}^{\infty}\frac{1}{m!}(\sum_{\alpha,\beta}|a_{\alpha,\beta}|A^{\alpha}B^{\beta})^{m}\left(\int_{|y-h|<2}|f(y)|^{p}w(y)dy\right)^{\frac{1}{p}}\\
&=C\exp(\sum_{\alpha,\beta}|a_{\alpha,\beta}|A^{\alpha}B^{\beta})\left(\int_{|y-h|<2}|f(y)|^{p}w(y)dy\right)^{\frac{1}{p}}.
\end{align*}
Thus
$$\|\widetilde{T}_{0}^{+}f\|_{L^{p}(w)}\leq C\|f\|_{L^{p}(w)}.$$

(c) implies (a):
Set $$\displaylines{
b(r)=\left\{\begin{array}{ll}
\displaystyle 1,&\quad r\in[0,1),\\
\displaystyle 0,&\quad r\in[1,+\infty)
.\end{array}\right.}$$
It is easy to see that $b(r)$ is a bounded variation function on $[0,+\infty)$.

Since the truncated operator $\widetilde{T}_{0}^{+}$ is a $(L^{p}(w), L^{p}(w))$ type operator, from Theorem \ref{Theorem 1.2}, the operator $T_{0}^{+}$ is a $(L^{p}(w), L^{p}(w))$ type operator. By the methods similar to the proof of (\ref{4.9}), the operator $T_{\infty}^{+}$ is a $(L^{p}(w), L^{p}(w))$ type operator. Therefore $T^{+}$ is a $(L^{p}(w), L^{p}(w))$ type operator, which implies Theorem \ref{Theorem 1.3}.

\qed

\subsection{Proof of Theorem \ref{Theorem 1.4}}\label{section3.3}

The proof of Theorem \ref{Theorem 1.4} can also be addressed by a double induction on the degree in $x$ and $y$ of the polynomial $P$.
Set
$$
P(x,y)=a_{kl}x^{k}y^{l}+R(x,y).
$$
Since our conclusion is clearly invariant under dialation by the proof of Theorem \ref{Theorem 1.2}, it is suitable to assume that $|a_{kl}|=1$.

If $k=0$, the conclusion holds from the result in \cite{AFM}. For general $P(x,y)$,
\begin{align*}
T_{\ast}^{+}f(x)
&\leq \sup_{0<\varepsilon<1} \left|\int_{x+\varepsilon}^{\infty}e^{iP(x,y)}K(x-y)f(y)dy\right|+\sup_{\varepsilon\geq1} \left|\int_{x+\varepsilon}^{\infty}e^{iP(x,y)}K(x-y)f(y)dy\right|\\
&\leq \sup_{0<\varepsilon<1} \left|\int_{x+\varepsilon}^{x+1}e^{iP(x,y)}K(x-y)f(y)dy|+|\int_{x+1}^{\infty}e^{iP(x,y)}K(x-y)f(y)dy\right|\\
&\,\,\,\,\,\,\,\,\,\,+\sup_{\varepsilon\geq1} \left|\int_{x+\varepsilon}^{\infty}e^{iP(x,y)}K(x-y)f(y)dy\right|\\
&=T_{\ast,0}^{+}f(x)+\left|\int_{x+1}^{\infty}e^{iP(x,y)}K(x-y)f(y)dy\right|+T_{\ast,\infty}^{+}f(x).
\end{align*}
Now, it suffices to prove that $T_{\ast,0}^{+}$ and $T_{\ast,\infty}^{+}$ are $(L^{p}(w), L^{p}(w))$ type operators. By the method similar to prove (\ref{4.5}), $T_{\ast,0}^{+}$ is $(L^{p}(w), L^{p}(w))$ type operators for $w\in A_{p}^{+}$ and the norm of $T_{\ast,0}^{+}$ depends on the total degree of $P(x,y)$, not on the coefficients of $P(x,y)$.

For the term $T_{\ast,\infty}f(x)$, there is a $J\in Z^{+}$ such that $2^{J-1}\leq \varepsilon <2^{J}$ allows the following to be shown:
\begin{align*}
T_{\ast,\infty}f(x)
&\leq \sup_{J\in Z^{+}} \left|\int_{2^{J-1}}^{2^{J}}\frac{|f(x-y)|}{|y|}dy\right|+\sup_{J\in Z^{+}}\sum_{j=J+1} \left|\int_{x+2^{j-1}}^{x+2^{j}}e^{iP(x,y)}K(x-y)f(y)dy\right|\\
&\leq M^{+}f(x)+\sum_{j=1}^{\infty} \left|\int_{x+2^{j-1}}^{2^{j}}e^{iP(x,y)}K(x-y)f(y)dy\right|.
\end{align*}
By the boundedness of $M^{+}$ and the method similar to prove (\ref{4.9}), we can derive the following:
$$\|T_{\ast,\infty}^{+}f\|_{L^{p}(w)}\leq C\|f\|_{L^{p}(w)}$$
where $C$ depends on the total degree of $P(x,y)$, not on the coefficients of $P(x,y)$.   \qed

{\bf Acknowledgements.}\quad  The authors thank the anonymous referees cordially for their valuable suggestions on this paper.



\begin{thebibliography}{99}

\bibitem{AC}
H. Aimar and R. Crescimbeni, On one-sided BMO and Lipschitz functions, \emph{Ann. Sc. Norm.
Super. Pisa Cl. Sci.}, 27(1998), 437--456.

\bibitem{AFM}
H. Aimar, L. Forzani and F. Mart\'{i}n-Reyes, On weighted
inequalities for one-sided singular integrals, \emph{Proc. Amer.
Math. Soc.}, 125(1997), 2057--2064.

\bibitem{AS}
K. Andersen and E. Sawyer, Weighted norm inequalities for the Riemann-Liouville and
Weyl fractional integral operators, \emph{Trans. Amer. Math. Soc.}, 308(1988), 547--558.

\bibitem{C}
A. Calder\'{o}n, Ergodic theory and translation invariant operators, \emph{Proc. Nat. Acad. Sci.
USA}, 59(1968): 349-353.

\bibitem{CM}
{R. Coifman and Y. Meyer},
Au del\'{a} des op\'{e}rateurs pseudo-diff\'{e}rentiels, Ast\'{e}risque 57, 1978.

\bibitem{DH}
{D. Deng and Y. Han}, Theory of $H^{p}$ space (in Chinese), Peking University Publishing House, 2001.

\bibitem{DS}
N. Dunford and J. Schwartz, Convergence almost everywhere of
operator averages, \emph{Proc. Nat. Acad. Sci. USA}, 41(1955),
229--231.

\bibitem{FS1}
C. Fefferman and E. Stein, Hardy spaces of several variables, \emph{Acta Math.}, 129(1972),
137--193.

\bibitem{FS2}
R. Fefferman and F. Soria, The spaces weak $H^{1}$, \emph{Studia Math.}, 85(1987),
1--16.

\bibitem{FL}
Z. Fu and S. Lu, One-sided Triebel-Lizorkin space and its applications(in Chinese), \emph{Sci.
Sin. Math.}, 41(2011), 43-52.

\bibitem{FLSS}
{Z. Fu, S. Lu, S. Sato and S. Shi}, On weighted weak type norm
inequalities for one-sided oscillatory singular integrals, \emph{
Studia Math.}, 207(2)(2011), 137--151.

\bibitem{FSL}
{Z. Fu, S. Lu, Y. Pan and S. Shi}, Weighted norm inequality for the
one-sided oscillatory singular operators, \emph{Abstr. Appl. Anal.}, 2014(2014), ID 291397, 7 pages.

\bibitem{G}
{L. Grafakos}, Classical and modern Fourier analysis, Pearson Education, 2004.

\bibitem{HLLW}
Y. Han, J. Li, G. Lu and P. Wang, $H^{p}\rightarrow H^{p}$ boundedness implies  $H^{p}\rightarrow L^{p}$ boundedness,
 \emph{Forum Math.}, 23(2011), 729--756.

\bibitem{HP}
Y. Hu and Y. Pan, Boundedness of oscillatory singular integrals on Hardy
spaces, \emph{Ark. Mat.} 30(1992), 311--320.

\bibitem{H}
Y. Hu, Oscillatory singular integrals on weighted Hardy
spaces, \emph{Studia Math.} 102(2)(1992), 145--156.

\bibitem{KPV1}
{C. Kenig, G. Ponce and L. Vega},
Oscillatory integrals and regularity of dispersive equations, \emph{Indiana Univ. Math. J.} 40(1991), 33--69.

\bibitem{KPV2}
{C. Kenig, G. Ponce and L. Vega},
Well-posedness and scattering results for the generalized
Korteweg-de Vries equation via the contraction principle, \emph{Comm. Pure Appl. Math.,} 4(1993), 527--620.

\bibitem{L1}
{S. Lu}, Four Lectures on real $H^{p}$ spaces, World Scientific Publishing, Singapore, 1995.

\bibitem{L2}
{S. Lu}, A class of oscillatory Integrals, \emph{Int. J. Appl. Math. Sci.}, 2(1)(2005), 42--58.

\bibitem{LDY}
{S. Lu, Y. Ding and D. Yan},
Singular integrals and Related Topics, World Scientific Publishing, Singapore, 2007.

\bibitem{LZ1}
{S. Lu and Y. Zhang},
Weighted norm inequality of a class of oscillatory singular operators, \emph{China. Sci. Bull.}, 37(1992), 9--13.

\bibitem{LZ2}
{S. Lu and Y. Zhang}, Criterion on $L^{p}$-boundedness for a class
of oscillatory singular integrals with rough kernels, \emph{Rev.
Mat. Iberoamericana}, 8(1992), 201--219.

\bibitem{Ma}
{F. Mart\'{i}n-Reyes}, New proofs of weighted inequalities for the
one-sided Hardy-Littlewood maximal functions, \emph{Proc. Amer.
Math. Soc.}, 117(1993), 691--698.

\bibitem{MOT}
{F. Mart\'{i}n-Reyes, P. Ortega and A. de la Torre}, Weights
for one-sided operators, \emph{Recent development in real and Harmonic analysis},
Applied and Numerical Harmonic Analysis, Birkh\"{a}user Boston, 2009.

\bibitem{MT2}
F. Mart\'{i}n-Reyes and A. de la Torre, One-sided BMO spaces, \emph{J. London Math. Soc.}, 9(1994),
529--542.

\bibitem{MR}
 K. Miller and B. Ross, An introduction to the fractional calculus and fractional
differential equations, New York: John Wiley Sons Inc., 1993.

\bibitem{Mu}
{B. Muckenhoupt}, Weighted norm inequalities for the Hardy maximal
functions, \emph{Trans. Amer. Math. Soc.}, 165(1972), 207--226.

\bibitem{OS}
{S. Ombrosi and C. Segovia}, One-sided singular integral operators on Calder\'{o}n-Zygumund-Hardy spaces, \emph{Rev. Un. Mat. Argentina}, 44(1)(2003), 17--32.

\bibitem{OST}
S. Ombrosi, C. Segovia and R. Testoni, An interpolation theorem between one-sided Hardy spaces, \emph{Ark. Mat.}, 44(2006), 335--348.


\bibitem{P}
Y. Pan, Hardy spaces and oscillatory integral operators,
\emph{Rev. Mat. Iberoamericana}, 7(1991), 55--64.

\bibitem{PS}
{D. Phong and E. Stein},
 Singular integrals related to the Radon transform and boundary value problems, \emph{Proc. Nat. Acad. Sci. USA}, 80(1983),7697--7701.

\bibitem{RS}
{F. Ricci and E. Stein}, Harmonic analysis on nilpotant groups and
singular integrals I: Oscillatory Integrals, \emph{J. Funct.
Anal.}, 73(1987), 179--194.

\bibitem{RiT}
{M. Riveros and A. de la Torre}, On the best ranges for $A_{p}^{+}$ and
$RH_{r}^{+}$, \emph{Czechoslovak Math. J.}, 126(2001), 285--301.

\bibitem{RS1}
{L. de Rosa and C. Segovia}, Dual spaces for one-sided weighted Hardy spaces, \emph{Rev. Un. Mat.
Argentina}, 40(1997), 49--71.

\bibitem{RS2}
{L. de Rosa and C. Segovia}, Weighted $H^{p}$ spaces for one-sided maximal functions,
\emph{Contemp. Math.}, 189(1995), 161--183.

\bibitem{RS2.5}
{L. de Rosa and C. Segovia}, One-sided Littlewood-Paley theory,
\emph{J. Fourier Anal. Appl.}, 3(1)(1997), 933--957.

\bibitem{RS3}
{L. de Rosa and C. Segovia}, Equivalence of norms in one-sided $H^{p}$ spaces,
\emph{Collect. Math.}, 53(1)(2002), 1--20.

\bibitem{Sa}
{S. Sato},
Weighted weak type (1,1) estimates for oscillatory singular
integrals, \emph{Studia  Math.} 47(2001), 1--17.

\bibitem{Saw}
{E. Sawyer}, Weighted inequalities for the one-sided
Hardy-Littlewood maximal function, \emph{Trans. Amer. Math. Soc.},
297(1986), 53--61.

\bibitem{SF}
{S. Shi and Z. Fu}, Estimates of some operators on one-sided weighted
Morrey spaces, \emph{Abstr. Appl. Anal.}, 2013(2013), ID 829218, 9 pages.

\bibitem{St}
{E. Stein}, Harmonic Analysis: real-variable methods, orthogonality, and oscillatory
integrals, Princeton Univ. Press,
Princeton, 1993.

\bibitem{SW}
{M. Stein and G. Weiss},
Interpolation of operators with change of measures, \emph{Trans. Amer. Math. Soc.}, 87(1958),  159--172.
\end{thebibliography}
\end{document}